\documentclass[twoside,reqno,A4]{amsart}

\usepackage{latexsym,rotate,eucal,cite}
\usepackage{amsmath,amsthm,amssymb,amsxtra}
\theoremstyle{plain}
\newtheorem{theorem}{Theorem}
\newtheorem{lemma}{Lemma}[section]
\newtheorem{proposition}{Proposition}

\theoremstyle{remark}

 \numberwithin{equation}{section}

\begin{document} 
\noindent Differential Geometry and Its Applications.

\vskip.4in
\title{$\delta(3)$-ideal null 2-type hypersurfaces in Euclidean spaces}
\dedicatory{\textsc{
     Bang-Yen Chen$^1$ and Yu Fu$^2$ \\[1mm]}
       ${}^1$ Department of Mathematics,
        Michigan State University, \\ East Lansing, Michigan 48824-1027, USA\\
        ${}^2$ School of Mathematics and
        Quantitative Economics,\\
        Dongbei University of Finance and Economics,\\
        Dalian 116025, P. R. China}

\thanks{Email addresses: bychen@math.msu.edu, yufudufe@gmail.com}


\begin{abstract}
In the theory of finite type submanifolds, null 2-type submanifolds
are the most simple ones, besides 1-type submanifolds (cf. e.g.,
\cite{chen1988,chenbook1984}). In particular, the classification
problems of null 2-type hypersurfaces are quite interesting and of
fundamentally important. In this paper, we prove that every
$\delta(3)$-ideal null 2-type hypersurface in a Euclidean space has
constant mean curvature and constant scalar curvature.
\end{abstract}

\maketitle
\markboth{B.Y.Chen and Y.Fu}{null 2-type hypersurfaces in Euclidean
space}
\thispagestyle{empty}

\section{Introduction}
\hspace*{\parindent}
In the late 1970s,  the first author \cite{chenbook1984} introduced
the theory of finite type submanifolds in order to derive the best possible
estimates of the total mean curvature of a compact submanifold of
Euclidean space in terms of spectral geometry. Since then the theory of finite type has been
developed greatly (see \cite{chenbook1984} for more details).

Let $x: M^n\rightarrow\mathbb{E}^m$ be an
isometric immersion of an $n$-dimensional connected  Riemannian manifold
$M^n$ into the Euclidean $m$-space $\mathbb{E}^m$. Denote by $\Delta$ the
Laplace operator with respect to the induced Riemannian metric. A
submanifold $M^n$ of $\mathbb E^m$ is said to be of {\em finite type}
\cite{chenbook1984,chen1991,chen1993,chenlue} if its position vector
field $x$  admits the following spectral decomposition:
\begin{eqnarray*}
x=c_0+x_1+\cdots+x_k,
\end{eqnarray*}
where $c_0$ is a constant vector and $x_1, \ldots, x_k$ are
non-constant maps satisfying $$\Delta x_i=\lambda_ix_i,\;\;\; i=1,
\ldots, k.$$ In particular, if all of the eigenvalues $\lambda_1, \ldots,
\lambda_k$ are mutually different, then the submanifold $M^n$ is
said to be of $k$-type. In particular, if one of $\lambda_1, \ldots, \lambda_k$
is zero, then $M^n$ is said to be of null $k$-type. Obviously, null $k$-type immersions occur only when $M^n$ is non-compact.

It is well-known that a 1-type submanifold of a Euclidean space
$\mathbb E^m$ is either a minimal submanifold of $\mathbb E^m$ or a
minimal submanifold of a hypersphere in $\mathbb E^m$.

By the definition, null 2-type submanifolds are the most simple ones
of finite type submanifolds besides 1-type submanifolds. After
choosing a coordinate system on $\mathbb E^m$ with $c_0$ as its
origin, we have the following simple spectral decomposition for a
null 2-type submanifold $M^n$:
\begin{eqnarray}\label{1.1}
x=x_1+x_2,\quad \Delta x_1=0,\quad \Delta x_2=a x_2,
\end{eqnarray}
where $a$ is non-zero real number.  According to the well-known Beltrami
formula $\Delta x=-n\overrightarrow{H}$,  \eqref{1.1} implies the following equation
\begin{eqnarray}\label{1.2}\Delta\overrightarrow{H}=a\overrightarrow{H},
\end{eqnarray} where $\overrightarrow{H}$ is the mean curvature vector.
 Biharmonic
submanifolds in $\mathbb E^m$ are defined by the equation
$\Delta\overrightarrow{H}=0$. A result from \cite{chen19882} states that a Euclidean submanifold  satisfying \eqref{1.2} is either
biharmonic, or of 1-type, or of null 2-type.

Due to its simplicity, the first author proposed  in 1991 the following problem \cite[Problem 12]{chen1991}:
\vskip.05in

{\em ``Determine all submanifolds of Euclidean spaces which are of
null 2-type. In particular, classify null 2-type hypersurfaces in
Euclidean spaces.''}
\vskip.05in

The first result on null 2-type submanifolds was obtained by the
first author  in 1988  by proving that every null 2-type surface in
$\mathbb E^3$ is an open portion of a circular cylinder
$S^1\times\mathbb R$ \cite{chen1988}. Later on, Ferr\'{a}ndez and
Lucas \cite{ferrandez1991} showed that a null 2-type hypersurface in
$\mathbb E^{n+1}$ with at most two distinct principal curvatures is
a spherical cylinder $S^p\times\mathbb R^{n-p}$. In 1995, Hasanis
and Vlachos \cite{hasanisvlachos} proved that null 2-type
hypersurfaces in $\mathbb E^4$ have constant mean curvature and
constant scalar curvature (see also \cite{defever1995}). In 2012,
the first author and Garray \cite{chengaray2014} proved that
$\delta(2)$-ideal null 2-type hypersurfaces in Euclidean space are
spherical cylinders. In addition, $\delta(2)$-ideal
$H$-hypersurfaces of a Euclidean space were classified by the first
and Munteanu in \cite{chenMunteanu2013}. In \cite{Tu}, Turgay
determined $H$-hypersurfaces in a Euclidean space with three
distinct principal curvatures. Very recently, the second author
proved in \cite{fuyu2014} that null 2-type hypersurfaces with at
most three distinct principal curvatures have constant mean
curvature and constant scalar curvature.  Null 2-type submanifolds
with codimension $\geq 2$ have been studied, among others, in
\cite{chenlue,dursun2005,dursun2007}. For the most recent surveys in
this field, we refer the readers to \cite{chen2014,chenbook1984}.

In 1991, the first author posted in \cite{chen1991} the following challenging conjecture:
\vskip.05in

{\em The only biharmonic
submanifolds of Euclidean spaces are the minimal ones}.
\vskip.05in

 Since then  biharmonic
submanifolds become a very active research subject (cf. \cite{chen2013,chen2014,chenbook1984}). However, this biharmonic conjecture remains open.

For an $n$-dimensional Riemannian manifold $M^n$ with $n\geq3$ and an integer $r\in [2,n-1]$, the
first author introduced the $\delta$-invariant $\delta(r)$ by
\begin{eqnarray}
\delta(r)(p)=\tau(p)-{\rm inf}\,\tau(L_p^r),\
\end{eqnarray}
where $\tau(p)$ is the scalar curvature of $M^n$ and ${\rm inf}\,
\tau(L_p^r)$ is the function assigning to the point $p$ the infimum of
the scalar curvature for $L_p^r$ running over all $r$-dimensional linear
subspaces in $T_pM^n$ (cf. \cite{chenbook2011} for details).

For any isometric immersion of a Riemannian $n$-manifold $M^n\, (n\geq 3)$ into a
Euclidean $m$-space $\mathbb E^m$, the first author proved the following universal inequality
\cite{C00}:
\begin{eqnarray}\label{1.4}
\delta(r)\leq\frac{n^2(n-r)}{2(n-r+1)}H^2,
\end{eqnarray} where $H^2$ is the squared mean curvature.

Since the inequality \eqref{1.4} is a very general and sharp
inequality, it is a very natural and interesting problem to classify
submanifolds satisfying the equality case of  \eqref{1.4}
identically. Following \cite{chenbook2011}, such a submanifold in
$\mathbb E^m$ is called $\delta(r)$-$ideal$.  $\delta (2)$ and
$\delta(3)$-ideal submanifolds are the simplest ideal submanifolds.
Investigating the classification problems of $\delta (2)$-ideal and
$\delta(3)$-ideal submanifolds is quite interesting. In particular,
many interesting results on $\delta(2)$-ideal submanifolds has been done by many
geometers since the invention of $\delta$-invariants  (see
\cite{chenbook2011} for details, and recent work
\cite{MunteanuVrancken, antic2007}). In contrast, few results on
$\delta(3)$-ideal submanifolds are known.

In this paper, we investigate $\delta(3)$-ideal null 2-type
hypersurfaces in Euclidean space. Our main result states that
every $\delta(3)$-ideal null 2-type hypersurface in a Euclidean
space  has constant mean curvature and constant scalar
curvature.

\section{Preliminaries}
\noindent Let $x: M^n\rightarrow\mathbb{E}^{n+1}$ be an isometric
immersion of a hypersurface $M^n$ into $\mathbb{E}^{n+1}$. Denote
the Levi-Civita connections of $M^n$ and $\mathbb{E}^{n+1}$ by
$\nabla$ and $\bar\nabla$, respectively. Let $X$ and $Y$ be
vector fields tangent to $M^n$ and let $\xi$ be a unit normal
vector field. Then the Gauss and Weingarten formulas are given
respectively by (cf. \cite{chenbook2011, chenbook1973})
\begin{eqnarray}
\bar\nabla_XY&=&\nabla_XY+h(X,Y),\label{l23}\\
\bar\nabla_X\xi&=&-AX,\label{l16}
\end{eqnarray}
where $h$ is the second fundamental form, and $A$ is the shape
operator (or the Weingarten operator). It is well known that the second
fundamental form $h$ and the shape operator $A$ are related by
\begin{eqnarray}
\langle h(X,Y),\xi\rangle=\langle AX,Y\rangle.
\end{eqnarray}
The mean curvature vector field $\overrightarrow{H}$ is given by
\begin{eqnarray}
\overrightarrow{H}=\left(\frac{1}{n}\right){\rm Tr}\,h,
\end{eqnarray} where ${\rm Tr}\,h$ is the trace of $h$.
The Gauss and Codazzi equations are given respectively by
\begin{eqnarray*}
R(X,Y)Z=\langle AY,Z\rangle AX-\langle AX,Z\rangle AY,
\end{eqnarray*}
\begin{eqnarray*}
(\nabla_{X} A)Y=(\nabla_{Y} A)X,
\end{eqnarray*}
where $R$ is the curvature tensor,\ $\left<\;\,,\;\right>$ the inner product, and $\nabla A$ is defined by
\begin{eqnarray}
(\nabla_XA)Y=\nabla_X(AY)-A(\nabla_XY)
\end{eqnarray}
for all $X, Y, Z$ tangent to $M$.
Let us put $\overrightarrow{H}=H\xi$, where $H$ denotes the mean
curvature. The scalar curvature $\tau$ is then given by
\begin{eqnarray}\label{2.6}
\tau=\frac{1}{2}(n^2H^2-{\rm Tr}\, A^2).
\end{eqnarray}
By identifying the tangent and the normal parts of the condition
\eqref{1.2}, we have the following necessary and sufficient conditions for $M^n$
to be of null 2-type in $\mathbb E^{n+1}$ (cf. e.g. \cite{chen1996,chenbook1984,chengaray2014,chenMunteanu2013}).

\begin{proposition}
Assume that $M^n$ is not of 1-type. A hypersurface $M^n$ in a
Euclidean $(n+1)$-space $\mathbb E^{n+1}$ is null 2-type if and only
if
\begin{equation}\label{2.7}
\begin{cases}
\Delta H+H \,{\rm Tr}\, A^2=aH,\\
2A\,{\rm grad}H+n\, H{\rm grad}H=0,
\end{cases}
\end{equation}
where the Laplace operator $\Delta$ acting on scalar-valued function
$f$ is given by
\begin{eqnarray}\label{2.8}
\Delta f=-\sum_{i=1}^n(e_ie_i f-\nabla_{e_i}e_i f)
\end{eqnarray}
for an orthonormal local tangent frame $\{e_1,\ldots,e_n\}$ on
$M^n$.\end{proposition}

We need the following result from \cite[Theorem 13.7]{chenbook2011}.
\begin{proposition}
Let $M^n$ be a hypersurface in the Euclidean space $\mathbb
E^{n+1}$. Then
\begin{eqnarray}
\delta(3)\leq\frac{n^2(n-3)}{2(n-2)}H^2,
\end{eqnarray}
where the equality case holds at a point $p$ if and only if there is
an orthonormal basis $\{e_1,\ldots, e_n\}$ at $p$ such that the
shape operator at $p$ satisfies
\begin{eqnarray}\label{2.10}
A=\left( \begin{array}{cccccc} \alpha&0&0&0&\ldots&0\\
0&\beta&0&0&\ldots&0\\0&0&\gamma&0&\ldots&0\\0&0&0&\alpha+\beta+\gamma&\ldots&0\\
\vdots&\vdots&\vdots&\vdots&\ddots&\vdots\\0&0&0&0&\ldots&\alpha+\beta+\gamma
\end{array} \right)
\end{eqnarray}
for some functions $\alpha, \beta, \gamma$ defined on $M^n$. If this
happens at every point, we call $M^n$ a $\delta(3)$-ideal
hypersurface in $\mathbb E^{n+1}$.
\end{proposition}

\section{$\delta(3)$-ideal null 2-type hypersurfaces}
\noindent In this section, we determine $\delta(3)$-ideal null
2-type hypersurfaces $M^n$ in Euclidean space $\mathbb E^{n+1}$ with
$n\geq4$. We assume that $M^n$ is not of 1-type, hence $M^n$ is not
minimal.

If the mean curvature $H$ is constant, the first equation of
\eqref{2.7} implies that the length of the second fundamental form
is also constant. Combining these with \eqref{2.6} shows that the
scalar curvature $\tau$ is constant as well.
Hence, in the following text we suppose that the mean curvature $H$
is non-constant.

\begin{lemma} \label{L:3.1}
Let $M^n$ be a $\delta(3)$-ideal hypersurface satisfying the second
equation of \eqref{2.7} in $\mathbb E^{n+1}$ with non-constant mean
curvature $H$. If the shape operator of $M^n$ satisfies
\eqref{2.10}, then, up to reordering of $\alpha$, $\beta$ and
$\gamma$, with respect to a suitable orthonormal frame
$\{e_1,\ldots, e_n\}$ we have
\begin{eqnarray*}
\alpha=-\frac{n}{2}H~~{and}~~ \gamma=\frac{n^2}{2(n-2)}H-\beta.
\end{eqnarray*}
\end{lemma}
\begin{proof}
Let $M^n$ be a hypersurface satisfying the second equation in
\eqref{2.7} with the shape operator given by \eqref{2.10}. Also
assume that ${\rm grad}\,H$ is non-vanishing. Then one of its principal
curvatures $\lambda_1$ must be $-{nH}/{2}$ with multiplicity 1
and the corresponding principal direction is $e_1 ={\rm
grad}\,\lambda_1/|{\rm grad}\,\lambda_1|$. By taking into account
\eqref{2.10}, up to reordering $\alpha, \beta, \gamma$ we can
assume either $\lambda_1=\alpha$ or $\lambda_1=\alpha+\beta+\gamma$.
The former case gives the case of Lemma 3.1. In the latter case,
because the multiplicity of $\lambda_1$ is 1, we have $n=4$ which
implies $\alpha+\beta+\gamma=-2H$. However, from these equations and
\eqref{2.10} one can obtain
\begin{eqnarray*}
4H=\lambda_1+\lambda_2+\lambda_3+\lambda_4=2(\alpha+\beta+\gamma)=-4H,
\end{eqnarray*}
which implies $H=0$. This is a contradiction.
\end{proof}
According to Lemma \ref{L:3.1},  $e_1$ is parallel to ${\rm
grad}\,H$ and so \eqref{2.10} becomes
\begin{eqnarray}\label{3.1}
A=\mathrm{diag}(\lambda_1,\lambda_2,\lambda_3, \lambda_4, \ldots,
\lambda_n)
\end{eqnarray}
with $\lambda_1=-\frac{n}{2}H$, $\lambda_2=\beta$,
$\lambda_3=\frac{n^2}{2(n-2)}H-\beta$ and
$\lambda_4=\cdots=\lambda_n=\frac{n}{n-2}H$.

Denote by $c_1=-\frac{n}{2}$ and $c_2=\frac{n^2}{2(n-2)}$, and hence
$\frac{n}{n-2}=c_1+c_2$.
 Note that in our case $M^n$ has four distinct principal
curvatures. Thus we have
\begin{eqnarray}
\beta\neq c_1H, \;\frac{1}{2}c_2H, \; (c_1+c_2)H,\; (c_2-c_1)H,\; -c_1H.
\end{eqnarray}
Since the vector field $e_1$ is parallel to ${\rm grad}\,H$,
computing ${\rm grad}\,H=\sum_{i=1}^ne_i(H)e_i$ gives
\begin{eqnarray}\label{3.3}
e_1(H)\neq0,\quad e_i(H)=0, \quad 2\leq i\leq n.
\end{eqnarray}

Let us put $$\nabla_{e_i}e_j=\sum_{k=1}^n\omega_{ij}^ke_k,\;\;\;
1\leq i, j\leq n.$$ By computing $\nabla_{e_k}\! \langle
e_i,e_i\rangle=0$ and $\nabla_{e_k}\!\langle e_i,e_j\rangle=0$, we
find
\begin{eqnarray}
&&\omega_{ki}^i=0,\label{3.4}\\
&&\omega_{ki}^j+\omega_{kj}^i=0,\label{3.5}
\end{eqnarray}
for $i\neq j$ and  $1\leq i, j, k\leq n$. From \eqref{3.1},
\eqref{3.3} and \eqref{3.4}, the Codazzi equation reduces to
\begin{eqnarray}
e_i(\lambda_j)=(\lambda_i-\lambda_j)\omega_{ji}^j,\label{3.6}\\
(\lambda_i-\lambda_j)\omega_{ki}^j=(\lambda_k-\lambda_j)\omega_{ik}^j,\label{3.7}
\end{eqnarray}
where $i ,j ,k$ are mutually different and $1\leq i, j, k\leq n$.
From \eqref{3.3} we have
\begin{eqnarray*}
[e_i,e_j](H)=0,\quad 2\leq i, j\leq n, \quad i\neq j,
\end{eqnarray*}
which implies that
\begin{eqnarray}\label{3.8}
\omega_{ij}^1=\omega_{ji}^1, \quad 2\leq i, j\leq n, \quad i\neq j.
\end{eqnarray}
Compute $[e_1,e_i](H)=\big(\nabla_{e_1}e_i-\nabla_{e_i}e_1\big)(H)$
for $i=2, \ldots, n$. From \eqref{3.3} we find $\omega_{i1}^1=0$.
By choosing $j=1$, $i=2,  \ldots, n$ in \eqref{3.5}, and by
\eqref{3.3} we obtain $\omega_{1i}^1=0$ for $2\leq i\leq n$. Hence
we have
\begin{eqnarray}\label{3.9}
e_ie_1(H)=0,\quad 2\leq i\leq n.
\end{eqnarray}
Sequentially, by \eqref{3.9} the formula
$[e_1,e_i]\big(e_1(H)\big)=\big(\nabla_{e_1}e_i-\nabla_{e_i}e_1\big)\big(e_1(H)\big)$
gives
\begin{eqnarray}\label{3.10}
e_ie_1e_1(H)=0,\quad 2\leq i\leq n.
\end{eqnarray}
By choosing $j=1$ in \eqref{3.7}, we have
\begin{eqnarray*}
(\lambda_i-\lambda_1)\omega_{ki}^1=(\lambda_k-\lambda_1)\omega_{ik}^1,
\end{eqnarray*}
which together with \eqref{3.8} yields
\begin{eqnarray}\label{3.10}
\omega_{ij}^1=0,\quad 2\leq i, j\leq n, \quad i\neq j.
\end{eqnarray}
Combining \eqref{3.10} with \eqref{3.5} gives
\begin{eqnarray}\label{3.12}
\omega_{i1}^j=0, \quad 2\leq i, j\leq n, \quad i\neq j.
\end{eqnarray}
By choosing $k=2$ in \eqref{3.7}, we deduce that
\begin{eqnarray*}
(\lambda_i-\lambda_j)\omega_{2i}^j=(\lambda_2-\lambda_j)\omega_{i2}^j,
\end{eqnarray*}
which yields
\begin{eqnarray}\label{3.13}
\omega_{i2}^j=0, \quad 4\leq i, j\leq n, \quad i\neq j.
\end{eqnarray}
Similarly, we also have
\begin{eqnarray}\label{3.14} \omega_{i3}^j=0, \quad 4\leq i, j\leq n, \quad i\neq j.
\end{eqnarray}
 Now we state an important lemma for later
use.
\begin{lemma}\label{L:3.2}
Under the assumptions above, we have
\begin{eqnarray*} e_i(\beta)=0, \quad 2\leq i\leq n. \end{eqnarray*}
\end{lemma}
\begin{proof}
By applying \eqref{3.1}, we deduce from the Gauss equation that
$R(e_2,e_k)e_1=0$ for $3\leq k\leq n$. It follows from \eqref{3.12}
that
\begin{eqnarray*}
&&\nabla_{e_2}\nabla_{e_k}e_1=e_2(\omega_{k1}^k)e_k+\omega_{k1}^k\Big(\sum_{i=2}^n\omega_{2k}^ie_i\Big),\\
&&\nabla_{e_k}\nabla_{e_2}e_1=e_k(\omega_{21}^2)e_2+\omega_{21}^2\Big(\sum_{i=3}^{n}\omega_{k2}^ie_i\Big),\\
&&\nabla_{[e_2,e_k]}e_1=\sum_{i=2}^n(\omega_{2k}^i-\omega_{k2}^i)\omega_{i1}^ie_i.
\end{eqnarray*}
Hence, by the definition of curvature tensor, $\langle
R(e_2,e_k)e_1,e_2\rangle=0$ gives
\begin{eqnarray}\label{3.15}
e_k(\omega_{21}^2)=\omega_{2k}^2(\omega_{k1}^k-\omega_{21}^2),\quad
3\leq k\leq n.
\end{eqnarray}
In a similar way, by taking into account $\langle
R(e_3,e_k)e_1,e_3\rangle$, we find
\begin{eqnarray}\label{3.16}
e_k(\omega_{31}^3)=\omega_{3k}^3(\omega_{k1}^k-\omega_{31}^3),\quad
k=2~{\rm or}~ 4\leq k\leq n.
\end{eqnarray}
At this moment, by using \eqref{2.8}, \eqref{3.1}, \eqref{3.3},
\eqref{3.5} and \eqref{3.6}, the first equation of \eqref{2.7}
becomes
\begin{eqnarray}\label{3.17}
&&-e_1e_1(H)-\Big\{\frac{e_1(\beta)}{c_1H-\beta}+\frac{e_1(c_1H-\beta)}{(c_1-c_2)H+\beta}-(n-3)\frac{(c_1+c_2)e_1(H)}{c_2H}\Big\}e_1(H)\nonumber
\\&&+H\big\{c_1^2H^2+\beta^2+(c_2H-\beta)^2+(n-3)(c_1+c_2)^2H^2-a\big\}=0.
\end{eqnarray}
We first show that $e_3(\beta)=0$ and $e_2(\beta)=0$.

By using \eqref{3.5}, \eqref{3.6} and choosing $k=3$ in
\eqref{3.15}, \eqref{3.15} becomes
\begin{eqnarray}\label{3.18}
e_3\Big(\frac{e_1(\beta)}{c_1H-\beta}\Big)=
\frac{e_3(\beta)}{c_2H-2\beta}\Big(\frac{e_1(c_2H-\beta)}{(c_1-c_2)H+\beta}-\frac{e_1(\beta)}{c_1H-\beta}\Big),
\end{eqnarray}
which can be rewritten as
\begin{eqnarray*}
e_3e_1(\beta)=\Big\{-\frac{e_1(\beta)}{c_1H-\beta}+\frac{c_1H-\beta}{c_2H-2\beta}
\Big(\frac{e_1(c_2H-\beta)}{(c_1-c_2)H+\beta}-\frac{e_1(\beta)}{c_1H-\beta}\Big)\Big\}e_3(\beta).
\end{eqnarray*}
Hence
\begin{eqnarray}\label{3.19}
&&e_3\Big(\frac{e_1(c_1H-\beta)}{(c_1-c_2)H+\beta}\Big)\nonumber\\
&&=\frac{e_3(\beta)}{c_2H-2\beta}\Big(\frac{e_1(c_2H-\beta)}{(c_1-c_2)H+\beta}-\frac{e_1(\beta)}{c_1H-\beta}\Big)
\frac{-(c_1+c_2)H+3\beta}{(c_1-c_2)H+\beta}.
\end{eqnarray}
Since $e_3(\beta)\neq0$, by differentiating \eqref{3.17} along
$e_3$ and by applying \eqref{3.3}, \eqref{3.9}, \eqref{3.10}, \eqref{3.18} and
\eqref{3.19} we get
\begin{equation}\label{3.20}
\frac{e_1(c_2H-\beta)}{(c_1-c_2)H+\beta}-\frac{e_1(\beta)}{c_1H-\beta}+H(2\beta-c_2H)[(c_1-c_2)H+\beta]=0.
\end{equation}
Also, by differentiating \eqref{3.20} along $e_3$ and by using \eqref{3.18} and
\eqref{3.19} we obtain
\begin{equation}\label{3.21}
\Big(\frac{e_1(c_2H-\beta)}{(c_1-c_2)H+\beta}-\frac{e_1(\beta)}{c_1H-\beta}\Big)\frac{-2c_1H+2\beta}{(c_1-c_2)H+\beta}
+H(c_2H-2\beta)[(2c_1-3c_2)H+4\beta]=0.
\end{equation}
Thus \eqref{3.21} together with \eqref{3.20} implies
\begin{eqnarray*}
3H(2\beta-c_2H)^2=0,
\end{eqnarray*}
which is equivalent to $2\beta-c_2H=0$; namely,
\begin{eqnarray*}
\lambda_2=\lambda_3.
\end{eqnarray*}
This contradicts to our assumption. Hence, we must have
$e_3(\beta)=0$.
By applying \eqref{3.16} for $k=2$, a quite similar argument as the
above case yields $e_2(\beta)=0$.

In the remaining case, we will prove $e_k(\beta)=0$ for $k\geq 4$.

After applying \eqref{3.6}, equations \eqref{3.15} and \eqref{3.16} could be
rewritten as
\begin{eqnarray}\label{3.22}
e_k\Big(\frac{e_1(\beta)}{c_1H-\beta}\Big)=-
\frac{e_k(\beta)}{(c_1+c_2)H-\beta}\Big(\frac{(c_1+c_2)e_1(H)}{c_2H}+\frac{e_1(\beta)}{c_1H-\beta}\Big),\\
\label{3.23} e_k\Big(\frac{e_1(c_2H-\beta)}{(c_1-c_2)H+\beta}\Big)=
\frac{e_k(\beta)}{c_1H+\beta}\Big(\frac{(c_1+c_2)e_1(H)}{c_2H}+\frac{e_1(c_2H-\beta)}{(c_1-c_2)H+\beta}\Big).
\end{eqnarray}
\indent Assume $e_k(\beta)\neq0$. Then after differentiating \eqref{3.17}
along $e_k$ and by applying \eqref{3.3}, \eqref{3.9}, \eqref{3.10},
\eqref{3.22} and \eqref{3.23}, we have
\begin{eqnarray}\label{3.24}
&&\Big\{\frac{1}{(c_1+c_2)H-\beta}\Big(\frac{(c_1+c_2)e_1(H)}{c_2H}+\frac{e_1(\beta)}{c_1H-\beta}\Big)\nonumber\\&&
-\frac{1}{c_1H+\beta}\Big(\frac{(c_1+c_2)e_1(H)}{c_2H}
+\frac{e_1(c_2H-\beta)}{(c_1-c_2)H+\beta}\Big)
\Big\}e_1(H)\nonumber\\&&+2H(2\beta-c_2H)=0.
\end{eqnarray}
Now, by differentiating \eqref{3.24} along $e_k$, we derive from
\eqref{3.3}, \eqref{3.9}, \eqref{3.22} and \eqref{3.23} that
$4He_k(\beta)=0,$ which is a contradiction. Consequently, we complete the proof of Lemma \ref{L:3.2}.
\end{proof}

\begin{lemma} \label{L:3.3}
The connection coefficients of the hypersurfaces $M^n$ are expressed as
follows:
\begin{eqnarray*}
&&\nabla_{e_1}e_1=\nabla_{e_1}e_2=\nabla_{e_1}e_3=0;\;\;
\nabla_{e_1}e_i=\sum_{k=4}^n\omega_{1i}^ke_k,~4\leq i\leq n;\\
&&\nabla_{e_i}e_1=\omega_{i1}^ie_i,~2\leq i\leq n;\; \;\nabla_{e_2}e_2=\omega_{22}^1e_1; \;\; \nabla_{e_3}e_3=\omega_{33}^1e_1;\\
&&\nabla_{e_2}e_i=\sum_{k=3}^{n}\omega_{2i}^ke_k, ~3\leq i\leq n;
\\&& \nabla_{e_3}e_i=\omega_{3i}^2e_2+\sum_{k=4}^{n}\omega_{3i}^ke_k,
~i=2,~{\rm or}
~3\leq i\leq n;\\
 &&\nabla_{e_j}e_2=\omega_{j2}^3e_3,~4\leq j\leq n;\;\;  \nabla_{e_j}e_3=\omega_{j3}^2e_2,~4\leq j\leq n;
\\
&&\nabla_{e_i}e_j=\omega_{ij}^1\delta_{ij}e_1+\sum_{k=4}^{n}\omega_{ij}^ke_k,
~4\leq i, j\leq n.
\end{eqnarray*}
\end{lemma}
\begin{proof}
By using (3.3)-(3.8), (3.11)-(3.14) and Lemma \ref{L:3.2}, a
straightforward computation gives the connection coefficients of
$M^n$ in Lemma \ref{L:3.3}.
\end{proof}

By applying Gauss' equation and Lemma 3.3, we compute $\langle
R(X, Y)Z, W\rangle$. Then we obtain successively:
\begin{itemize}
\item $X=Z=e_1, Y=W=e_2$,
\begin{eqnarray}
e_1(\omega_{21}^2)+(\omega_{21}^2)^2=-c_1H\beta;\label{3.31}
\end{eqnarray}
\item $X=Z=e_1, Y=W=e_3$,
\begin{eqnarray}
e_1(\omega_{31}^3)+(\omega_{31}^3)^2=-c_1H(c_2H-\beta);\label{3.32}
\end{eqnarray}
\item $X=Z=e_1, Y=W=e_4$,
\begin{eqnarray}
e_1(\omega_{41}^4)+(\omega_{41}^4)^2=-c_1(c_1+c_2)H^2;\label{3.33}
\end{eqnarray}
\item $X=Z=e_j, Y=W=e_2$ for $4\leq j\leq n$,
\begin{eqnarray}
\omega_{jj}^1\omega_{21}^2-\omega_{2j}^3\omega_{j3}^2+(\omega_{j2}^3-\omega_{2j}^3)\omega_{3j}^2=(c_1+c_2)\beta
H;\label{3.34}
\end{eqnarray}
\item $X=Z=e_j, Y=W=e_3$ for $4\leq j\leq n$,
\begin{eqnarray}
\omega_{jj}^1\omega_{31}^3-\omega_{3j}^2\omega_{j2}^3+(\omega_{j3}^2-\omega_{3j}^2)\omega_{2j}^3=(c_1+c_2)
H(c_2H-\beta);\label{3.35}
\end{eqnarray}
\item $X=Z=e_3, Y=W=e_2$,
\begin{eqnarray}
\omega_{33}^1\omega_{21}^2-\sum_{j=4}^n\omega_{23}^j\omega_{3j}^2+\sum_{j=4}^n(\omega_{32}^j-\omega_{23}^j)\omega_{j3}^2=\beta(c_2H-\beta);\label{3.36}
\end{eqnarray}
\item $X=e_2, Y=e_j, Z=e_3, W=e_1$ for $4\leq j\leq n$,
\begin{eqnarray*}
\omega_{j3}^2\omega_{22}^1-\omega_{23}^j\omega_{jj}^1-(\omega_{2j}^3-\omega_{j2}^3)\omega_{33}^1=0\label{3.37};
\end{eqnarray*}
\item $X=e_3, Y=e_j, Z=e_2, W=e_1$ for $4\leq j\leq n$,
\begin{eqnarray*}
\omega_{j2}^3\omega_{33}^1-\omega_{32}^j\omega_{jj}^1-(\omega_{3j}^2-\omega_{j3}^2)\omega_{22}^1=0.\label{3.38}
\end{eqnarray*}
\end{itemize}
Moreover, it follows from \eqref{3.5} and \eqref{3.7} that
\begin{eqnarray}
&&\omega_{23}^j=-\omega_{2j}^3=-\frac{h_j}{\beta+c_1H}\label{3.39},\\
&&\omega_{32}^j=-\omega_{3j}^2=\frac{h_j}{\beta-(c_1+c_2)H},\label{3.40}\\
&&\omega_{j2}^3=-\omega_{j3}^2=\frac{h_j}{2\beta-c_2H},\label{3.41}
\end{eqnarray}
where $h_j$ $(4\leq j\leq n)$ are smooth functions defined on
$M^n$. Therefore (3.31)-(3.33) imply that
\begin{eqnarray*}
&&-\omega_{2j}^3\omega_{j3}^2=(\omega_{j2}^3-\omega_{2j}^3)\omega_{3j}^2,\\
&&-\omega_{3j}^2\omega_{j2}^3=(\omega_{j3}^2-\omega_{3j}^2)\omega_{2j}^3,\\
&&-\omega_{23}^j\omega_{3j}^2=(\omega_{32}^j-\omega_{23}^j)\omega_{j3}^2.
\end{eqnarray*}
By combining \eqref{3.34}, \eqref{3.35} and \eqref{3.36} with the above
three equations, we get
\begin{eqnarray}
&&\omega_{jj}^1\omega_{21}^2-2\omega_{2j}^3\omega_{j3}^2=(c_1+c_2)\beta
H, \quad 4\leq j\leq n,\label{3.42}\\
&&\omega_{jj}^1\omega_{31}^3-2\omega_{3j}^2\omega_{j2}^3=(c_1+c_2)
H(c_2H-\beta), \quad 4\leq j\leq n,\label{3.43}\\
&&\omega_{33}^1\omega_{21}^2-2\sum_{j=4}^n\omega_{23}^j\omega_{3j}^2=\beta(c_2H-\beta).\label{3.44}
\end{eqnarray}
Since $\omega_{44}^1=\omega_{55}^1=\cdots=\omega_{nn}^1$ and
$\omega_{2j}^3\omega_{j3}^2+\omega_{3j}^2\omega_{j2}^3+\omega_{23}^j\omega_{3j}^2=0,$
we derive from (3.34)-(3.36) that
\begin{eqnarray}
&&(n-3)\omega_{jj}^1\omega_{21}^2+(n-3)\omega_{jj}^1\omega_{31}^3+\omega_{33}^1\omega_{21}^2\label{3.45}\\
&&=(n-3)(c_1+c_2)c_2H^2+\beta(c_2H-\beta),\nonumber
\end{eqnarray}
By using \eqref{3.6}, equations (3.25)-(3.27) can be rewritten,
respectively,  as
\begin{eqnarray}
&&e_1e_1(\beta)-c_1\omega_{21}^2e_1(H)+2(c_1H-\beta)(\omega_{21}^2)^2=-c_1H\beta(c_1H-\beta),\label{3.46}\\
&&e_1e_1(c_2H-\beta)-c_1\omega_{31}^3e_1(H)+2\big\{(c_1-c_2)H+\beta\big\}(\omega_{31}^3)^2\label{3.47}\\
&&=-c_1H(c_2H-\beta)\big\{(c_1-c_2)H+\beta\big\},\nonumber\\
&&(c_1+c_2)e_1e_1(H)-c_1\omega_{41}^4e_1(H)-2c_2H(\omega_{41}^4)^2=c_1c_2(c_1+c_2)H^3.\label{3.48}
\end{eqnarray}
Also, it follows from \eqref{3.6} that
\begin{eqnarray}
&&(c_1H-\beta)\omega_{21}^2+\big\{(c_1-c_2)H+\beta\big\}\omega_{31}^3=c_2e_1(H),\label{3.49}\\
&&-c_2H\omega_{41}^4=(c_1+c_2)e_1(H).\label{3.50}
\end{eqnarray}
Eliminating $e_1e_1(\beta)$ between \eqref{3.46} and \eqref{3.47}
and applying \eqref{3.45}, \eqref{3.49} and \eqref{3.50} give
\begin{eqnarray}
&&c_2^2e_1e_1(H)+\big\{2(n-3)(c_1+c_2)(2c_1-c_2)+c_2(2c_2-c_1)\big\}(\omega_{21}^2+\omega_{31}^3)e_1(H)\nonumber\\
&&=\big\{2(n-3)(2c_1-c_2)(c_1+c_2)c_2^2-c_1c_2^2(c_1-c_2)\big\}H^3\nonumber\\
&&+2c_2^2(c_1-c_2)H^2\beta-2c_2(c_1-c_2)H\beta^2.\label{3.51}
\end{eqnarray}
Moreover, eliminating $(\omega_{41}^4)^2$ in \eqref{3.48} by
\eqref{3.50} we obtain
\begin{eqnarray}
(c_1+c_2)e_1e_1(H)+(c_1+2c_2)\omega_{41}^4e_1(H)=c_1c_2(c_1+c_2)H^3.\label{3.52}
\end{eqnarray}
Note that \eqref{3.17} takes the form
\begin{eqnarray}\label{3.53}
&&-e_1e_1(H)-\big\{\omega_{21}^2+\omega_{31}^3+(n-3)\omega_{41}^4\big\}e_1(H)
\\&&+\big\{c_1^2+c_2^2+(n-3)(c_1+c_2)^2\big\}H^3-2c_2H^2\beta+2H\beta^2-aH=0.\nonumber
\end{eqnarray}
By substituting $c_1=-\frac{n}{2}$ and $c_2=\frac{n^2}{2(n-2)}$ into
(3.43)-(3.45), we get
\begin{eqnarray}
&&e_1e_1(H)-\frac{9n^2-50n+48}{n^2}(\omega_{21}^2+\omega_{31}^3)e_1(H)\label{3.54}\\
&&+\frac{n^2(7n^2-29n+26)}{2(n-2)^2}H^3+\frac{2n(n-1)}{n-2}H^2\beta-\frac{4(n-1)}{n}H\beta^2=0,\nonumber\\
&&e_1e_1(H)+\frac{n+2}{2}\omega_{41}^4e_1(H)+\frac{n^3}{4(n-2)}H^3=0,\label{3.55}\\
&&-e_1e_1(H)-\big\{\omega_{21}^2+\omega_{31}^3+(n-3)\omega_{41}^4\big\}e_1(H)\label{3.56}
\\&&+\frac{n^2(n+2)}{2(n-2)}H^3-\frac{n^2}{n-2}H^2\beta+2H\beta^2-aH=0.\nonumber
\end{eqnarray}
Eliminating $e_1e_1(H)$, equations (3.46)-(3.48) reduce to
\begin{eqnarray*}
&&\Big\{\frac{9n^2-50n+48}{n^2}(\omega_{21}^2+\omega_{31}^3)+\frac{n+2}{2}\omega_{41}^4\Big\}e_1(H)\\
&&=\frac{n^2(13n^2-56n+52)}{4(n-2)^2}H^3+\frac{2n(n-1)}{n-2}H^2\beta+\frac{4(2n-3)}{n}H\beta^2,\nonumber\\
&&\Big\{\omega_{21}^2+\omega_{31}^3+\frac{n-8}{2}\omega_{41}^4\Big\}e_1(H)\\
&&=\frac{n^2(3n+4)}{4(n-2)}H^3-\frac{n^2}{n-2}H^2\beta+2H\beta^2-aH.\nonumber
\end{eqnarray*}
These equations imply
\begin{eqnarray}
&&\frac{2(2n^3+31n^2-112n+96)}{n^2}\omega_{41}^4e_1(H)\label{3.57}\\
&&=\frac{7n^4-56n^3+86n^2+152n-192}{2(n-2)^2}H^3-\frac{11n^2-52n+48}{n-2}H^2\beta\nonumber\\
&&+\frac{2(5n^2-44n+48)}{n^2}H\beta^2-\frac{a(9n^2-50n+48)}{n^2}H,\nonumber\\
&&\frac{2(2n^3+31n^2-112n+96)}{n^2}\big(\omega_{21}^2+\omega_{31}^3\big)e_1(H)\label{3.58}\\
&&=\frac{n^2(5n^3-82n^2+256n-200)}{4(n-2)^2}H^3+\frac{n(3n^2-16n+16)}{2(n-2)}H^2\beta\nonumber\\
&&+\frac{3n^2-40n+48}{n}H\beta^2+\frac{a(n+2)}{2}H.\nonumber
\end{eqnarray}
After combining \eqref{3.58} with \eqref{3.50} and \eqref{3.45}, we get
\begin{eqnarray}
&&\frac{2n^3+31n^2-112n+96}{n(n-3)}\omega_{21}^2\omega_{31}^3\label{3.59}\\
&&=\frac{n^2(n^3-144n^2+480n-392)}{4(n-2)^2}H^2+\frac{n(n^3-56n^2+176n-144)}{2(n-2)(n-3)}H\beta\nonumber\\
&&+\frac{5n^3-18n^2+56n-48}{n(n-3)}\beta^2+\frac{a(n+2)}{2}H.\nonumber
\end{eqnarray}
Now taking into account \eqref{3.50}, \eqref{3.57} becomes
\begin{eqnarray}
&&-\frac{4(2n^3+31n^2-112n+96)}{n^3}\big(e_1(H)\big)^2\label{3.60}\\
&&=\frac{7n^4-56n^3+86n^2+152n-192}{2(n-2)^2}H^4-\frac{11n^2-52n+48}{n-2}H^3\beta\nonumber\\
&&+\frac{2(5n^2-44n+48)}{n^2}H^2\beta^2-\frac{a(9n^2-50n+48)}{n^2}H^2.\nonumber
\end{eqnarray}
After differentiating \eqref{3.60} with respect to $e_1$ and by using
\eqref{3.57}, \eqref{3.55}, \eqref{3.49} and \eqref{3.6}, we obtain
\begin{eqnarray}
L(H,\beta)\omega_{21}^2+M(H,\beta)\omega_{31}^3=0,\label{3.61}
\end{eqnarray}
where $L$ and $M$ take the form
\begin{eqnarray*}
&&L(H,\beta)=a_0H^4+a_1H^3\beta+a_2H^2\beta^2+a_3H\beta^3+a_4H^2+a_5H\beta,\\
&&M(H,\beta)=b_0H^4+b_1H^3\beta+b_2H^2\beta^2+b_3H\beta^3+b_4H^2+b_5H\beta
\end{eqnarray*}
for some constants $a_i$ and $b_i$ with respect to $n$.

Moreover, substituting \eqref{3.49} into \eqref{3.58} yields
\begin{eqnarray}
\big(\omega_{21}^2+\omega_{31}^3\big)\Big\{(c_1H-\beta)\omega_{21}^2+[(c_1-c_2)H+\beta]\omega_{31}^3\Big\}=N(H,\beta),\label{3.62}
\end{eqnarray}
where
\begin{eqnarray*}
N(H,\beta)=c_0H^3+c_1H^2\beta+c_2H\beta^2+c_3H
\end{eqnarray*}
for some constants $c_i$ with respect to $n$.

To eliminate $\omega_{21}^2$ and $\omega_{31}^3$ from \eqref{3.59},
\eqref{3.61} and \eqref{3.62}, a direct computation gives the
following equation of ninth degree involving $H$ and $\beta$
\begin{eqnarray}
&&c_{90}H^9+c_{81}H^8\beta+c_{72}H^7\beta^2+c_{63}H^6\beta^3+c_{54}H^5\beta^4
+c_{45}H^4\beta^5\label{3.63}\\
&&+c_{36}H^3\beta^6+c_{27}H^2\beta^7+c_{18}H\beta^8+c_{09}\beta^9+c_{70}H^7+c_{61}H^6\beta\nonumber\\
&&+c_{52}H^5\beta^2+c_{43}H^4\beta^3+c_{34}H^3\beta^4+c_{25}H^2\beta^5+c_{16}H\beta^6+c_{07}\beta^7\nonumber\\
&&+c_{50}H^5+c_{41}H^4\beta+c_{32}H^3\beta^2+c_{23}H^2\beta^3+c_{14}H\beta^4+c_{05}\beta^5\nonumber\\
&&+c_{30}H^3+c_{21}H^2\beta+c_{12}H\beta^2+c_{03}\beta^3=0,\nonumber
\end{eqnarray}
where the coefficients $c_{ij}$ ($i, j=0,\ldots,9$) are constants
concerning $n$.

Note that $\beta$ is not constant in general. In fact, if $\beta$ is
a constant, then \eqref{3.63} is an algebraic equation of $H$ with
constant coefficients. Thus the real function $H$ must be a
constant and hence the conclusion follows immediately.

Now, let us consider an integral curve of the vector field $e_1$ passing through $p=\gamma(t_0)$ as $\gamma(t), t\in I$.  Because $e_1(H),
e_1(\beta)\neq0$  and $e_i(H)=e_i(\beta)=0$ for $2\leq i\leq n$ according to Lemma \ref{L:3.2}, we may assume that $t=t(\beta)$ and $H=H(\beta)$ in
some neighborhood of $\beta_0=\beta(t_0)$.
It follows from \eqref{3.6}, \eqref{3.49} and \eqref{3.61} that
\begin{eqnarray}
\frac{dH}{d\beta}&=&\frac{dH}{dt}\frac{dt}{d\beta}=\frac{e_1(H)}{e_1(\beta)}\label{3.64}\\
&=&\frac{(c_1H-\beta)\omega_{21}^2+[(c_1-c_2)H+\beta]\omega_{31}^3}{c_2(c_1H-\beta)\omega_{21}^2}\nonumber\\
&=&\frac{1}{c_2}-\frac{[(c_1-c_2)H+\beta]L}{c_2(c_1H-\beta)M}.\nonumber
\end{eqnarray}
By differentiating \eqref{3.63} with respect to $\beta$, together
with \eqref{3.64}, we get another algebraic equation of twelfth
degree involving $H$ and $\beta$:
\begin{eqnarray}
&&b_{12,0}H^{12}+b_{11,1}H^{11}\beta+b_{10,2}H^{10}\beta^2+b_{93}H^9\beta^3+b_{84}H^8\beta^4
+b_{75}H^7\beta^5\label{3.65}\\
&&+b_{66}H^6\beta^6+b_{57}H^5\beta^7+b_{48}H^4\beta^8+b_{39}H^3\beta^9+b_{2,10}H^2\beta^{10}+b_{1,11}H\beta^{11}\nonumber\\
&&+b_{0,12}\beta^{12}+b_{10,0}H^{10}+b_{91}H^9\beta+b_{82}H^8\beta^2+b_{73}H^7\beta^3+b_{64}H^6\beta^4
\nonumber\\
&&+b_{55}H^5\beta^5+b_{46}H^4\beta^6+b_{37}H^3\beta^7+b_{28}H^2\beta^8+b_{19}H\beta^9+b_{0,10}\beta^{10}+b_{80}H^8\nonumber\\
&&+b_{71}H^7\beta+b_{62}H^6\beta^2+b_{53}H^5\beta^3+b_{44}H^4\beta^4+b_{35}H^3\beta^5+b_{26}H^2\beta^6+b_{17}H\beta^7\nonumber\\
&&+b_{08}\beta^8+b_{60}H^6+b_{51}H^5\beta+b_{42}H^4\beta^2+b_{33}H^3\beta^3+b_{24}H^2\beta^4+b_{15}H\beta^5\nonumber\\
&&+b_{06}\beta^6+b_{40}H^4+b_{31}H^3\beta+b_{22}H^2\beta^2+b_{13}H\beta^3+b_{04}\beta^4=0,\nonumber
\end{eqnarray}
where the coefficients $b_{ij}$ ($i, j=0,\ldots,12$) are some
constants concerning $n$.

We may rewrite \eqref{3.63} and \eqref{3.65} respectively in the
following forms
\begin{eqnarray}\label{3.66}
\sum_{i=0}^9q_i(H)\beta^i=0,\qquad \sum_{j=0}^{12}{\bar
q}_j(H)\beta^j=0,
\end{eqnarray}
where $q_i(H)$ and ${\bar q}_j(H)$ are polynomials functions of the mean curvature
function $H$. After eliminating $\beta$ between the two equations in
\eqref{3.66}, we obtain a non-trivial algebraic polynomial equation
of $H$ with constant coefficients. Hence we conclude that the real
function $H$ has to be a constant, which contradicts to our original
assumption. Therefore the mean curvature function $H$ must be
constant for any $\delta(3)$-ideal null 2-type hypersurfaces in a
Euclidean space.

Consequently, we have proved the following theorem.
\begin{theorem}
Every $\delta(3)$-ideal null 2-type hypersurface in a Euclidean
space must have constant mean curvature and constant scalar
curvature.
\end{theorem}

\section*{Acknowledgments}
The authors would like to thank the referees for the very valuable
suggestions and comments to improve the original version of the
paper. The second author is supported by the Mathematical Tianyuan
Youth Fund of China (No. 11326068), Project funded by China
Postdoctoral Science Foundation (No. 2014M560216) and the Excellent
Innovation talents Project of DUFE (No. DUFE2014R26).

\end{document}